\documentclass[12pt]{amsart}
\date{31 January 2015}
\usepackage{latexsym,amsmath,amsfonts,amscd,amssymb}
\usepackage{lscape}
\usepackage{array}
\usepackage[mathscr]{eucal} 
\usepackage{graphicx}
\usepackage{xypic}

\DeclareFontFamily{OT1}{rsfs}{}
\DeclareFontShape{OT1}{rsfs}{n}{it}{<->rsfs10}{}
\DeclareMathAlphabet{\curly}{OT1}{rsfs}{n}{it}

\theoremstyle{plain}  % default
\newtheorem{theorem}{Theorem}[section]

\newtheorem*{theorem*}{Theorem}
\newtheorem{corollary}[theorem]{Corollary}

\newtheorem{proposition}[theorem]{Proposition}

\newtheorem{definition}[theorem]{Definition}
\theoremstyle{remark}
\newtheorem{remark}[theorem]{Remark}
\numberwithin{equation}{section}
\newcommand{\suchthat}{\;\;:\;\;}
\renewcommand{\leq}{\leqslant}

\newcommand{\R}{\mathbb{R}}
\newcommand{\Z}{\mathbb{Z}}
\newcommand{\C}{\mathbb{C}}
\newcommand{\too}{\longrightarrow}
\newcommand{\cM}{\mathcal{M}}
\newcommand{\calR}{\mathcal{R}}
\newcommand{\bE}{{\bf{E}}}

\newcommand{\dbar}{\overline{\partial}}

\newcommand{\lra}{\longrightarrow}
\newcommand{\SL}{\mathrm{SL}}
\DeclareMathOperator{\ad}{ad}
\DeclareMathOperator{\Ad}{Ad}
\DeclareMathOperator{\im}{im}
\DeclareMathOperator{\Hom}{Hom}
\DeclareMathOperator{\Id}{Id}
\DeclareMathOperator{\Aut}{Aut}
\DeclareMathOperator{\Int}{Int}
\DeclareMathOperator{\Out}{Out}
\DeclareMathOperator{\Conj}{Conj}
\newcommand{\Mg}{\mathcal{M}^{\operatorname{gauge}}}

\renewcommand{\phi}{\varphi}

\newcommand{\lieh}{\mathfrak{h}}
\newcommand{\lieg}{\mathfrak{g}}
\newcommand{\liep}{\mathfrak{p}}
\newcommand{\lieq}{\mathfrak{q}}

\newcommand{\AAA}{{\curly A}}
\newcommand{\CCC}{{\curly C}}
\newcommand{\DDD}{{\curly D}}
\newcommand{\GGG}{{\curly G}}
\newcommand{\HHH}{{\curly H}}
\newcommand{\XXX}{\curly{X}}

\newcommand{\Real}{\mathop{{\fam0 Re}}\nolimits}
\newcommand{\Imaginary}{\mathop{{\fam0 Im}}\nolimits}

\newcommand{\kahler}{K\"{a}hler}

\newcommand{\bmu}{{\boldsymbol{\mu}}}
\begin{document}

%%%%%%%%%%%%%%%%%%%%%%%%%%%%%%%%%%%%%%%%%%%%%%%%%%%%%%%%%%%%%%%%%%%
\title[Anti-holomorphic involutions of the moduli of Higgs bundles]
{Anti-holomorphic involutions of the moduli spaces of Higgs bundles}
%%%%%%%%%%%%%%%%%%%%%%%%%%%%%%%%%%%%%%%%%%%%%%%%%%%%%%%%%%%%%%%%%%%

\author[I. Biswas]{Indranil Biswas}

\address{School of Mathematics, Tata Institute of Fundamental
Research, \\Homi Bhabha Road,\\ Bombay 400005, India}
\email{indranil@math.tifr.res.in}

\author[O. Garc{\'\i}a-Prada]{Oscar Garc{\'\i}a-Prada}
\address{Instituto de Ciencias Matem\'aticas \\
  CSIC \\ Nicol\'as Cabrera, 13--15 \\ 28049 Madrid \\ Spain}
\email{oscar.garcia-prada@icmat.es}

\thanks{The first author acknowledges the support of the J. C. Bose Fellowship.
The second author is partially supported by the Ministerio de Econom\'ia y 
Competitividad of Spain through Grant MTM2010-17717 and Severo Ochoa 
Excellence Grant}



\begin{abstract}
We study anti-holomorphic involutions of the moduli space of 
$G$-Higgs bundles over a  compact Riemann surface $X$, 
where $G$ is a complex semisimple Lie group. These involutions are 
defined by fixing anti-holomorphic involutions on both $X$ and $G$.
We analyze the fixed point locus in the moduli space and their relation with
representations of the orbifold fundamental group of $X$ equipped with
the anti-holomorphic involution. We also study the relation with branes.
This generalizes work by Biswas--Garc\'{\i}a-Prada--Hurtubise and 
Baraglia--Schaposnik.
\end{abstract}

\maketitle

%%%%%%%%%%%%%%%%%%%%%%
\section{Introduction}
%%%%%%%%%%%%%%%%%%%%%%

Let $G$ be a complex semisimple affine algebraic group with Lie algebra $\lieg$. 
Let $X$ be  a compact connected Riemann surface.
A  $G$--Higgs bundle over $X$ is a pair $(E,\varphi)$, where
$E$ is a holomorphic principal  $G$-bundle over $X$ and
$\varphi$ is a holomorphic section of $E(\lieg)\otimes K$ with $E(\lieg)$ being
the vector bundle associated to
$E$ for the adjoint action of $G$ on $\lieg$ and $K$ being the canonical line
bundle on $X$. We consider  the   moduli space of polystable
$G$-Higgs bundles $\mathcal{M}(G)$. This has the structure of a hyper-K\"ahler
manifold outside the singular locus.

Let $\alpha:X\lra X$ and  $\sigma:G\lra G$ be anti-holomorphic involutions.
We define the two involutions (see Section \ref{section-involutions} for details)
\begin{equation}\label{involution-a}
   \begin{aligned}
\iota(\alpha,\sigma)^{\pm}\,:\, \cM(G) & \longrightarrow \cM(G) \\
(E\, ,\varphi) & \longmapsto (\alpha^*\sigma(E)\, ,\pm \alpha^\ast\sigma(\varphi)).
  \end{aligned}
\end{equation}

The goal of this paper is to describe the fixed points of these involutions.
The fixed points are given by the image of  moduli spaces of $G$-Higgs bundles
satisfying a reality condition determined by $\alpha$ and $\sigma$, and an
element $c\in Z_2^\sigma$, where $Z$ is the center of $G$ and $Z_2^\sigma$ is
the group of elements of order two in $Z$ fixed by $\sigma$. 
For the  involution $\iota(\alpha,\sigma)^+$, these are the moduli space of pseudo-real 
Higgs bundles considered in \cite{BGH}, to which we refer here as
$(\alpha,\sigma,c,+)$-pseudo-real $G$-Higgs bundles.
 For $\iota(\alpha,\sigma)^-$, the
reality condition on the bundle $E$ is the same as for 
$\iota(\alpha,\sigma)^+$, but the different sign on $\varphi$ gives  a
different reality condition on the moduli space of Higgs bundles,
defining objects that we call $(\alpha,\sigma,c,-)$-pseudo-real $G$-Higgs bundles.
When the element $c\in Z_2^\sigma$ is trivial we call these objects real
$G$-Higgs bundles.

The involution $\iota(\alpha,\sigma)^-$ is studied 
by Baraglia-Schaposnik \cite{baraglia-schaposnik} 
when $\sigma$ is the anti-holomorphic involution
$\tau$ corresponding to a compact real form of $G$ (see also \cite{HWW}). 
In \cite{baraglia-schaposnik2}, they consider the involutions 
$\iota(\alpha,\sigma)^+$ obtained as a result of composing
$\iota(\alpha,\tau)^-$  with the holomorphic involution $\iota^-(\theta)$
of $\cM(G)$ defined by $\iota^-(E,\varphi)=(\theta(E),-\theta(\varphi))$,
where $\theta$ is the holomorphic involution of $G$ given by
$\theta=\sigma\tau$ (here one takes a compact conjugation $\tau$ 
commuting with $\sigma$). In fact, if we consider the involutions 

\begin{equation}\label{involution-h}
   \begin{aligned}
\iota(\theta)^{\pm}\,:\, \cM(G) & \longrightarrow \cM(G) \\
(E\, ,\varphi) & \longmapsto (\theta(E)\, ,\pm \theta(\varphi)),
  \end{aligned}
\end{equation}
one has
$$
\iota(\alpha,\sigma)^\pm= \iota^\mp(\theta) \circ \iota^-(\alpha,\tau).
$$ 
The involutions (\ref{involution-h}) have been studied in  
\cite{garcia-prada1,garcia-prada,GR}.

In the language of branes
\cite{kapustin-witten}, the fixed
points of $\iota(\alpha,\sigma)^+$ are $(A,A,B)$--branes, while the fixed
points  of $\iota(\alpha,\sigma)^-$ are $(A,B,A)$--branes. What these mean
is that the fixed points of $\iota(\alpha,\sigma)^-$ are complex
Lagrangian submanifolds with respect to the complex structure $J_2$ defined
on $\cM(G)$ by the complex structure of $G$, while the
fixed points of $\iota(\alpha,\sigma)^+$  are complex
Lagrangian submanifolds with respect to the complex structure $J_3\,=\,J_1J_2$
obtained by combining $J_2$ with the natural complex structure $J_1$ defined 
on the moduli space of Higgs bundles for the Riemann
surface $X$. The study of these branes is of great interest in connection 
with mirror symmetry and the Langlands correspondence in the theory of Higgs
bundles (see \cite{kapustin-witten,hitchin-mirror,baraglia-schaposnik,baraglia}).

We then identify these involutions in the moduli space of representations
of the fundamental group of $X$ in $G$, and describe the fixed points
corresponding to the $(\alpha,\sigma,c,\pm)$-pseudo-real $G$-Higgs bundles
in terms
of representations of the orbifold fundamental group of $(X,\alpha)$ in
a group whose underlying set is $G\times \Z/2\Z$. The group structure on
$G\times \Z/2\Z$ is constructed using the element
$c\,\in\, Z^\sigma_2$ and an action of $\Z/2\Z$ on $G$
which depends on the sign of the pseudo-reality condition; more precisely,
this action is given by the conjugation $\sigma$ in the ``$+$'' case, and
the action of $\theta\,=\,\sigma\tau$ in the ``$-$'' case, where 
$\tau$ is a compact conjugation commuting with $\sigma$. 
When $c$ is trivial we obtain the semi-direct products of 
$G$ with $\Z/2\Z$ for the action $\sigma$. 

The results of this paper have a straightforward generalization to the case
in which $G$ is reductive. In this situation  the fundamental group of $X$
is replaced by its universal central extension. 

%%%%%%%%%%%%%%%%%%%%%%%%%%%%%%%%%%%%%%%%%%%%%%%%%%%%%%%%%%%%%%%%%%%%%%%%%%%%%%%%%%%%%%%%%%%%%%
\section{$G$-Higgs bundles and representations of the fundamental group}\label{sec:higgs-reps}
%%%%%%%%%%%%%%%%%%%%%%%%%%%%%%%%%%%%%%%%%%%%%%%%%%%%%%%%%%%%%%%%%%%%%%%%%%%%%%%%%%%%%%%%%%%%%%

%%%%%%%%%%%%%%%%%%%%%%%%%%%%%%%%%%%%%%%%%%%%%%%%%%%%%%%%%%%%%%
\subsection{Moduli space of $G$-Higgs bundles}\label{moduli}
%%%%%%%%%%%%%%%%%%%%%%%%%%%%%%%%%%%%%%%%%%%%%%%%%%%%%%%%%%%%%

Let $G$ be a complex semisimple affine algebraic group. Its Lie algebra
will be denoted by $\lieg$. Let $X$ be an irreducible
smooth projective curve defined over $\C$, equivalently, it is a 
compact connected Riemann surface. Let $g_{_X}$ be the genus of $X$; we assume
that $g_{_X}\, \geq\, 2$. The canonical line bundle of $X$ will be denoted
by $K$. For a principal $G$-bundle $E$, let $E(\lieg)\,:=\,
E\times^G \lieg$ be the adjoint vector bundle for $E$.

A {\bf $G$-Higgs bundle} over $X$ is a pair $(E,\varphi)$, where
$E$ is a holomorphic principal $G$-bundle $E$ over $X$ and
$\varphi$ is a  holomorphic section of $E(\lieg)\otimes K$.
Two $G$-Higgs bundles $(E,\varphi)$ and $(F,\psi)$ are isomorphic if there is
a holomorphic isomorphism of principal $G$-bundles $f\,:\,E\,\longrightarrow\, F$
such that the induced isomorphism
$$\Ad(f)\otimes \Id_K\,:\, E(\lieg)\otimes K\,\longrightarrow\, F(\lieg)\otimes K$$
sends $\varphi$ to $\psi$.

There are notions of (semi)stability and polystability for $G$-Higgs bundles 
(see \cite{bradlow-garcia-prada-mundet,garcia-prada-gothen-mundet, BS} 
for example).
A $G$-Higgs bundle $(E,\varphi)$ is said to be
{\bf stable} (respectively, {\bf semistable}) if for every parabolic subgroup
$P\subset G$, every holomorphic reduction $\sigma:E_P\,\longrightarrow\, E$ of $E$ to $P$
such that $$\varphi\,\in\, H^0(X,\,E_P(\mathfrak{p})\otimes K)\,\subset\,
H^0(X,\,E(\lieg)\otimes K)
$$
and every strictly antidominant character $\chi$ of $P$, we
have that $\deg E_P(\chi)\,>\, 0$ (respectively, $\deg E_P(\chi)\,\geq\, 0$). 
A Higgs bundle $(E,\varphi)$ is {\bf polystable} if it is semistable and for
every $P$, every reduction and every $\chi$ as above such that 
$\deg E_P(\chi)\,=\,0$, there is a holomorphic reduction $E_L\,\subset\, E$ to a
Levi subgroup $L\,\subset\, P$ such that $\varphi\,
\in \,H^0(X,\,E_L(\mathfrak{l})\otimes K)$.

Let $\mathcal{M}(G)$ denote the {\bf moduli space of semistable
$G$-Higgs bundles} of fixed topological type. This moduli space has 
the structure of a complex normal quasiprojective variety of dimension
$\dim G(g_{_X}-1)$.

%%%%%%%%%%%%%%%%%%%%%%%%%%%%%%%%%%%%%%%%%%%%%%%%%%%%
\subsection{$G$-Higgs bundles and Hitchin equations}
\label{section-hitchin-equations}
%%%%%%%%%%%%%%%%%%%%%%%%%%%%%%%%%%%%%%%%%%%%%%%%%%%%

As above, let $G$ be a complex semisimple affine algebraic group. Let 
$H\subset G$ be a maximal compact subgroup.  Let $(E,\varphi)$ be a
$G$-Higgs bundle over a compact Riemann surface $X$. By a slight abuse
of notation, we shall denote the $C^\infty$-objects underlying $E$ and
$\varphi$ by the same symbols. In particular, the Higgs field can be
viewed as a $(1,0)$-form $\varphi \in
\Omega^{1,0}(E(\lieg))$ with values in $E(\lieg)$. Let
$$
\tau\,\colon\,
\Omega^{1,0}(E(\lieg)) \,\longrightarrow\, \Omega^{0,1}(E(\lieg))
$$
be the isomorphism induced by the
compact conjugation of $\lieg$ (with respect to $H$) combined with the complex
conjugation on complex $1$-forms. Given a $C^\infty$ reduction of
structure group $h$ of the principal $G$-bundle $E$ to $H$, we denote by $F_h$ the
curvature of the unique connection compatible with $h$ and the holomorphic
structure on $E$; see \cite[pp. 191--192, Proposition 5]{At} for the connection.

\begin{theorem} \label{higgs-hk}
There is a reduction $h$ of structure group of $E$ from $G$ to $H$
that satisfies the Hitchin equation
  $$
  F_h -[\varphi,\tau(\varphi)]= 0 
  $$
  if and only if $(E,\varphi)$ is polystable.
\end{theorem}

Theorem \ref{higgs-hk} was proved by
Hitchin \cite{hitchin:1987}  for $G\,=\,\SL(2,\C)$, and in
\cite{simpson,simpson:1992, BS} for the general case. 

\begin{remark}
When  $G$ is reductive the equation in Theorem \ref{higgs-hk} is replaced
by the equation $$F_h -[\varphi,\tau(\varphi)]\,=\, c\omega\, ,$$ where $\omega$ is a
K\"ahler form on $X$ and $c$ is an element in the center of the Lie algebra of $G$,
which is determined by the topology of $E$. 
\end{remark}

{}From the point of view of moduli spaces it is convenient
to fix a $C^\infty$ principal $H$--bundle
$\bE_H$ and study the moduli space of solutions to \textbf{Hitchin's equations}
for a pair $(A\, ,\varphi)$ consisting of a $H$--connection $A$ on $\bE_H$ and
a section $\varphi\,\in\, \Omega^{1,0}(X,\bE_H(\lieg))$:
\begin{equation}\label{hitchin}
\begin{array}{l}
F_A -[\varphi,\tau(\varphi)]\,=\, 0\\
\dbar_A\varphi\,=\, 0\, .
\end{array}
\end{equation}
Here $d_A$ is the covariant derivative associated to $A$, and
$\dbar_A$ is the $(0,1)$ part of $d_A$. The $(0,1)$ part of $d_A$ defines a holomorphic
structure on $\bE_H$. The gauge group $\HHH$ of $\bE_H$ acts on the
space of solutions and the moduli space of solutions is
$$
\Mg(G):= \{ (A,\varphi)\;\;\mbox{satisfying}\;\;
(\ref{hitchin})\}/\HHH.
$$

Now,
Theorem \ref{higgs-hk} can be reformulated as follows.

\begin{theorem} \label{hk}
There is a homeomorphism
$$
\cM(G)\,\cong\, \Mg(G)\, .
$$
\end{theorem}

To explain this correspondence we interpret the moduli
space of $G$-Higgs bundles in terms of pairs $(\dbar_E, \varphi)$ consisting
of a $\dbar$-operator (holomorphic structure) $\dbar_E$
on the $C^\infty$ principal $G$-bundle $\bE_{G}$ obtained from
$\bE_H$ by the extension of structure group $H\,\hookrightarrow\, G$, and
$\varphi\in \Omega^{1,0}(X,\bE_G(\lieg))$ satisfying $\dbar_E\varphi\,=\,0$.
Such pairs are in one-to-one correspondence with  $G$-Higgs bundles $(E,\varphi)$,
where $E$ is the holomorphic $G$-bundle defined by the operator
$\dbar_E$ on $\bE_G$. The equation $\dbar_E\varphi\,=\,0$ is equivalent
to the condition that $\varphi\in H^0(X,E(\lieg)\otimes K)$. 
The moduli space of polystable $G$-Higgs bundles $\cM_d(G)$ can now
be identified with the orbit space
$$
\{ (\dbar_E,\varphi)\;\;:\;\; \dbar_E\varphi=0\;\;\mbox{which are polystable}\}/
\GGG\, ,
$$
where $\GGG$ is the gauge group of $\bE_G$, which is in fact
the complexification of $\HHH$.
Since  there is a one-to-one correspondence between
$H$-connections on $\bE_H$ and $\dbar$-operators on $\bE_{G}$,
the correspondence given in Theorem \ref{hk} can be reformulated
by saying that in the $\GGG$--orbit of a polystable $G$-Higgs
bundle $(\dbar_{E_0},\varphi_0)$ we can find another Higgs bundle
$(\dbar_E,\varphi)$
whose corresponding pair $(d_A,\varphi)$ satisfies the Hitchin equation
$F_A -[\varphi,\tau(\varphi)]\,=\, 0$ with this pair
$(d_A,\varphi)$ being unique up to $H$-gauge transformations.

%%%%%%%%%%%%%%%%%%%%%%%%%%%%%%%%%%%%%%%%%%%%%%
\subsection{Higgs bundles and representations}
\label{section-reps}
%%%%%%%%%%%%%%%%%%%%%%%%%%%%%%%%%%%%%%%%%%%%%%

Fix a base point $x_0\, \in\, X$.
By a \textbf{representation} of $\pi_1(X,x_0)$ in
$G$ we mean a homomorphism $\pi_1(X,x_0) \,\longrightarrow\, G$.
After fixing a presentation of $\pi_1(X,x_0)$, the set of all such homomorphisms,
$\Hom(\pi_1(X,x_0),\, G)$, can be identified with the subset
of $G^{2g_{_X}}$ consisting of $2g_{_X}$-tuples
$(A_{1},B_{1}, \cdots, A_{g_{_X}},B_{g_{_X}})$ satisfying the algebraic equation
$\prod_{i=1}^{g_{_X}}[A_{i},B_{i}] \,=\, 1$. This shows that
$\Hom(\pi_1(X,x_0),\, G)$ is a complex algebraic variety.

The group $G$ acts on $\Hom(\pi_1(X,x_0),G)$ by conjugation:
\[
(g \cdot \rho)(\gamma) \,=\, g \rho(\gamma) g^{-1}\, ,
\]
where $g \,\in\, G$, $\rho \,\in\, \Hom(\pi_1(X,x_0),G)$ and
$\gamma\,\in \,\pi_1(X,x_0)$. If we restrict the action to the subspace
$\Hom^+(\pi_1(X, x_0),\,G)$ consisting of reductive representations,
the orbit space is Hausdorff.  We recall that a \textbf{reductive representation}
is one whose composition with the adjoint representation in $\mathfrak g$
decomposes as a direct sum of irreducible representations.
This is equivalent to the condition that the Zariski closure of the
image of $\pi_1(X,x_0)$ in $G$ is a reductive group. Define the
{\bf moduli space of representations} of $\pi_1(X,x_0)$ in $G$ to be the orbit space
\[
\mathcal{R}(G) = \Hom^{+}(\pi_1(X,x_0),G) /G\, .
\]
For another point $x'\, \in\, X$, the fundamental groups
$\pi_1(X,x_0)$ and $\pi_1(X,x')$ are identified by an isomorphism unique up to
an inner automorphism. Consequently, $\mathcal{R}(G)$ is independent of the choice of
the base point $x_0$.

One has the following (see e.g. \cite{goldman:1984}, \cite{Si}).

\begin{theorem}
The moduli space $\calR(G)$ has the structure of a normal complex variety. Its smooth
locus is equipped with a holomorphic symplectic form.
\end{theorem}

Given a representation $\rho\colon\pi_{1}(X,x_0) \,\longrightarrow\,
G$, there is an associated flat principal $G$-bundle on
$X$, defined as
$$
  E_{\rho} \,=\, \widetilde{X}\times^{\rho}G\, ,
$$
where $\widetilde{X} \,\longrightarrow\, X$ is the universal cover
associated to $x_0$ and $\pi_{1}(X, x_0)$ acts
on $G$ via $\rho$.
This gives in fact an identification between the set of equivalence classes
of representations $\Hom(\pi_1(X),G) / G$ and the set of equivalence classes
of flat principal $G$-bundles, which in turn is parametrized by
the (nonabelian) cohomology set $H^1(X,\, G)$. 

We have the following:

\begin{theorem}\label{na-Hodge}
There is a homeomorphism
$\mathcal{R}(G) \,\cong\, \mathcal{M}(G)$. 
\end{theorem}

The moduli spaces $\mathcal{M}(G)$ and  $\mathcal{R}(G)$ are sometimes
referred as the {\bf Dolbeault} and {\bf Betti} moduli spaces, respectively.

The proof of Theorem \ref{na-Hodge} is the combination of two
existence theorems for gauge-theoretic equations. To explain this, let
$\bE_G$ be, as above, a $C^\infty$ principal $G$-bundle over $X$ and ${\bf E}_H$
a $C^\infty$ reduction of structure group of it to $H$. Every $G$--connection
$D$ on $\bE_G$ decomposes uniquely as
$$
D=d_A + \psi,
$$
where $d_A$ is an $H$-connection on $\bE_H$ and
$\psi\in \Omega^1(X,\bE_H(\sqrt{-1}\lieh))$.  Let
$F_A$ be the curvature of $d_A$.
We consider the following set of equations for the pair $(d_A,\psi)$:
\begin{equation}\label{harmonicity}
\begin{array}{l}
F_A +\frac{1}{2}[\psi,\psi]\,=\, 0\\
d_A\psi\,=\,0  \\
d_A^\ast\psi\,=\,0\, .
\end{array}
\end{equation}
These equations are invariant under the action of $\HHH$, the gauge group of
$\bE_H$. A theorem of Corlette \cite{corlette}, and
Donaldson \cite{donaldson} for $G=\SL(2,\C)$,  says the following.
\begin{theorem}\label{corlette} There is a homeomorphism
between 
$$
\{\mbox{Reductive $G$-connections}\;\; D \suchthat
F_D=0\}/\GGG
$$
and
$$
\{(d_A,\psi)\;\;\mbox{satisfying}\;\;
(\ref{harmonicity})\}/\HHH.
$$
\end{theorem}

The first two equations in (\ref{harmonicity}) are equivalent
to the flatness of $D=d_A+\psi$, and Theorem \ref{corlette}
simply says that in the $\GGG$-orbit of a reductive flat $G$-connection
$D_0$ we can find a flat $G$-connection $D=\widetilde{g}(D_0)$ such that if we
write $D\,=\,d_A+\psi$, the
additional condition $d_A^\ast\psi\,=\,0$ is satisfied. This can be interpreted
more geometrically in terms of the reduction  $h\,=\,\widetilde{g}(h_0)$ of $\bE_G$
to a principal $H$-bundle obtained by the action of $\widetilde{g}\,\in\,\GGG$ on $h_0$.
The equation $d_A^\ast\psi\,=\,0$ is equivalent to the harmonicity of the
$\pi_1(X)$-equivariant map $\widetilde X \,\too\, G/H$ corresponding to
the new reduction of structure group $h$.

To complete the argument, leading to Theorem \ref{na-Hodge}, we just need
Theorem \ref{higgs-hk} and the following simple result.

\begin{proposition}\label{prop:circle}
The correspondence $(d_A,\varphi)\,\longmapsto\, (d_A,\psi:=\varphi-\tau(\varphi))$
defines  a homeomorphism
$$
\{(d_A,\varphi)\;\;\mbox{satisfying}\;\;
(\ref{hitchin})\}/\HHH\cong
\{(d_A,\psi)\;\;\mbox{satisfying}\;\;
(\ref{harmonicity})\}/\HHH.
$$
\end{proposition}

%%%%%%%%%%%%%%%%%%%%%%%%%%%%%%%%%%%%%%%%%%%%%%%%%%%%%%%%%
\subsection{The moduli space as a hyper-K\"ahler quotient}
%%%%%%%%%%%%%%%%%%%%%%%%%%%%%%%%%%%%%%%%%%%%%%%%%%%%%%%%%

We will see now that the moduli space $\cM(G)$ 
has a hyper-K\"ahler structure. For this, recall first that
a hyper-K\"ahler manifold is a differentiable manifold $M$ equipped with a 
Riemannian metric $g$ and complex structures $J_i$, $i=1,2,3$ 
satisfying the quaternion relations $J_i^2=-I$, 
$J_3=J_1J_2\,=\, -J_2J_1$, $J_2\,=\, -J_1J_3\,=\, J_3J_1$ and $J_1\,=\, J_2J_3
\,=\, -J_3J_2$ such that if we define $\omega_i(\cdot, \cdot)\,=\,g(J_i\cdot, \cdot)$, then
$(g,J_i,\omega_i)$ is a K\"ahler structure on $M$.
Let $\Omega_i$ denote the holomorphic symplectic structure on $\cM(G)$ with respect
to the complex structure $J_i$. In fact, $\Omega_1\,=\, \omega_2+\sqrt{-1}\omega_3$,
$\Omega_2\,=\, \omega_3+\sqrt{-1}\omega_1$ and $\Omega_3\,=\, \omega_1+\sqrt{-1}\omega_2$.

One way to understand the non-abelian Hodge theory correspondence 
mentioned above
is through the analysis of the hyper-K\"ahler 
structure  of the moduli spaces  involved.  We explain how these can 
be obtained as hyper-K\"ahler quotients. For this, 
let $\bE_G$ be a smooth principal $G$-bundle over $X$, and let $\bE_H$ be a fixed
reduction of $\bE_G$ to the maximal compact subgroup $H$.  
The set $\AAA$  of $H$-connections on $\bE_H$ is an  affine space modelled 
on  $\Omega^1(X,\bE(\lieh))$. Via the Chern correspondence $\AAA$ is in
one-to-one correspondence  with the set $\CCC$ of holomorphic structures on 
$\bE_{G}$ \cite[pp. 191--192, Proposition 5]{At}, which is an affine space modelled on
$\Omega^{0,1}(X,\bE_{G}(\lieg))$. Let us denote 
$\Omega=\Omega^{1,0}(X,\bE_G(\lieg))$. 
We  consider $\XXX=\AAA\times \Omega$. 
Via the identification $\AAA\cong\CCC$,
we have for $\alpha\in \Omega^{0,1}(X,\bE_{G}(\lieg))$ and 
$\psi\in\Omega^{1,0}(X,\bE_G(\lieg))$ the following three
complex structures
on $\XXX$:
$$
\begin{array}{ccc}
 J_1(\alpha,\psi)& = & (\sqrt{-1}\alpha,\sqrt{-1}\psi)\\
 J_2(\alpha,\psi)& = & (-\sqrt{-1}\tau(\psi),\sqrt{-1}\tau(\alpha))\\
 J_3 (\alpha,\psi)& = & (\tau(\psi),-\tau(\alpha)),
\end{array}
$$
where $\tau$ is the conjugation on $\lieg$ defining its
compact form $\lieh$ (determined fiber-wise by the reduction to
$\bE_H$), combined with the complex conjugation on complex $1$-forms. 

One has also a Riemannian metric $g$ defined on $\XXX$:\, 
for $\alpha\in \Omega^{0,1}(X,\bE_G(\lieg))$ and $\psi \in 
\Omega^{1,0}(X,\bE_G(\lieg))$,
$$
g((\alpha,\psi),(\alpha,\psi))\,=\,-2\sqrt{-1}\int_X B(\tau(\alpha), \alpha)
+
B(\psi, \tau(\psi))\, ,
$$
where $B$ is the Killing form.

Clearly, $J_i$, $i=1,2,3$, satisfy the quaternion relations, and define
a hyper-K\"ahler structure on $\XXX$, with K\"ahler forms
$\omega_i(\cdot,\cdot)=g(J_i\cdot,\cdot)$, $i=1,2,3$.
As shown in \cite{hitchin:1987}, the action of the gauge group $\HHH$ on 
$\XXX$ preserves the 
hyper-K\"ahler structure and there are  moment maps
given by
$$
\mu_1(A,\varphi)=F_A - [\varphi,\tau(\varphi)],\;\;\;
\mu_2(A,\varphi)=\Real(\dbar_E\varphi),\;\;\;
\mu_3(A,\varphi)=\Imaginary(\dbar_E\varphi).
$$
We have that ${\bmu}^{-1}({0})/\HHH$, where ${\bmu}=(\mu_1,\mu_2,\mu_3)$
is the moduli space of solutions to the Hitchin equations (\ref{hitchin}).
In particular, if we consider the irreducible solutions
(equivalently, smooth)
${\bmu}_*^{-1}({0})$ we have that
$$
{\bmu}_*^{-1}({0})/\HHH$$
is a hyper-K\"ahler manifold which, by 
Theorem \ref{hk},  is homeomorphic to the subvariety of smooth
points in  moduli space
$\cM(G)$ of stable $G$-Higgs bundles with the  topological
class of $\bE_G$.

Let us now see how the moduli of harmonic flat connections 
on $\bE_H$ can be realized as a hyper-K\"ahler\ quotient. 
Let $\DDD$ be the set of $G$-connections on $\bE_G$. 
This is an affine space modelled on 
$\Omega^1(X,\bE_G(\lieg))=\Omega^0(X,T^\ast X\otimes_\R\bE_{G}(\lieg))$.
The space $\DDD$ has a complex structure $I_1\,=\,1\otimes\sqrt{-1}$,
which comes from the complex structure of the
bundle. Using the complex structure of $X$ we have also
the complex structure $I_2\,=\,\sqrt{-1}\otimes \tau$.
We can finally consider the complex structure $I_3\,=\,I_1 I_2$.

The reduction to $H$ of the $G$-bundle $\bE_G$ together with a 
Riemannian metric in the conformal class of $X$ defines a flat
Riemannian metric $g_\DDD$ on $\DDD$ which is \kahler\ for the above 
three complex structures.
Hence $(\DDD,g_\DDD,I_1,I_2,I_3)$ is also a hyper-K\"ahler 
manifold. As in the previous case, the action of the gauge group $\HHH$
on $\DDD$ preserves the hyper-K\"ahler
structure  and there are  moment maps
$$
\mu_1(D)=d_A^\ast\psi,\;\;\;
\mu_2(D)=\im(F_D),\;\;\;
\mu_3(D)=\Real(F_D),
$$
where $D= d_A  +\psi$ is the decomposition of $D$ 
defined by 
$$
\bE_G(\lieg)=\bE_H(\lieh)\oplus \bE_H(\sqrt{-1}\lieh).
$$

Hence the moduli space of solutions to the harmonicity 
equations (\ref{harmonicity}) is the hyper-K\"ahler quotient defined by 
$$
{\bmu}^{-1}(0)/\HHH,
$$
where ${\bmu}=(\mu_1,\mu_2,\mu_3)$.
The homeomorphism  between the moduli spaces of solutions to the Hitchin
and the harmonicity equations 
is induced from the affine map 
$$
\begin{array}{ccc}
\AAA\times \Omega & \lra &\DDD\\
(d_A,\varphi)&\longmapsto & d_A  + \varphi -\tau(\varphi).
\end{array}
$$
One can see easily, for example, that this map sends
$\AAA\times\Omega$ with complex structure $J_2$ to $\DDD$ with complex
structure $I_1$ (see \cite{hitchin:1987}).

Now, Theorems  \ref{hk}  and \ref{corlette} 
can be regarded as existence theorems, establishing the non-emptiness
of the hyper-K\"ahler quotient, obtained by focusing on different
complex structures. For Theorem \ref{hk} one gives a special
status to the complex 
structure $J_1$. Combining the symplectic forms determined by  $J_2$
and $J_3$ one has  the $J_1$-holomorphic symplectic form
$\omega_c=\omega_2 +\sqrt{-1}\omega_3$ on $\AAA\times\Omega$. The
gauge group $\GGG=\HHH^\C$ acts on $\AAA\times \Omega$ preserving $\omega_c$.
The symplectic quotient construction can also be extended to the
holomorphic situation (see e.g. \cite{kobayashi:1987}) to obtain
the holomorphic symplectic quotient 
$\{(\dbar_E,\varphi)\;\;:\;\;\dbar_E\varphi=0\}/\GGG $.
What Theorem \ref{hk} says is that for a class 
$[(\dbar_E,\varphi)]$ in this quotient
to have a representative (unique up to $H$-gauge) satisfying
$\mu_1=0$ it is necessary  and sufficient that the pair 
$(\dbar_E, \varphi)$ be polystable. 
This identifies the hyper-K\"ahler quotient to 
the set of equivalence classes of polystable $G$-Higgs bundles on $\bE_G$.
If one now takes $J_2$ on $\AAA\times \Omega$ or equivalently 
$\DDD$ with $I_1$ and argues in a similar way, one gets Theorem 
\ref{corlette} identifying the  hyper-K\"ahler quotient to 
the set of equivalence classes of reductive  flat
connections on $\bE_G$.

%%%%%%%%%%%%%%%%%%%%%%%%%%%%%%%%%%%%%%%%%%%%%%%%%%%%%%%%%%%
\section{Real $G$-Higgs bundles}\label{g-theta-higgs}
%%%%%%%%%%%%%%%%%%%%%%%%%%%%%%%%%%%%%%%%%%%%%%%%%%%%%%%%%%%%

%%%%%%%%%%%%%%%%%%%%%%%%%%%%%%%%%%%%%%%%%%%%%%%%%%%%%%%%%%%%%%%
\subsection{Involutions and conjugations of complex Lie groups}
\label{realforms-group}
%%%%%%%%%%%%%%%%%%%%%%%%%%%%%%%%%%%%%%%%%%%%%%%%%%%%%%%%%%%%%%%

Let $G$ be a Lie group. We define 
$$
\Int(G):=\{f\in\Aut(G)\;\;|\;\; f(h)=ghg^{-1},\;\; \mbox{for every $h\in
  G$}\}.
$$
We have that $\Int(G)=\Ad(G)$.

We define the group of outer automorphisms of $G$ as
$$
\Out(G):=\Aut(G)/\Int(G).
$$

We have a sequence
\begin{equation}\label{outer-extension-group}
1 \lra \Int(G)\lra \Aut(G) \lra \Out(G)\lra 1.
\end{equation}

It is well-known that if $G$ is a connected complex reductive group then the extension 
(\ref{outer-extension-group}) splits (see \cite{de-siebenthal}).

Let $G$ be a complex Lie group and let $G_\R$ be the underlying
real Lie group. We will say that a  real Lie subgroup $G_0\subset G_\R$ 
is a {\bf real form} of $G$ if it is the fixed point set of a 
conjugation (anti-holomorphic involution)  $\sigma$ of $G$. 

Now, let $G$ be simple. A compact real form always exists. 
This follows from the fact that for a simple group there  is a maximal 
compact subgroup $U\subset G$, such that $U^\C=G$. From this we can define
a conjugation $\tau\,:\, G\,\too\, G$ such that $G^\tau=U$.
Let $\Conj(G)$ be the set of conjugations (i.e., anti-holomorphic involutions)
of $G$. We can define the following  equivalence
relations in $\Conj(G)$:
$$
\sigma \sim \sigma'  \;\;\mbox{if there is}\;\; \alpha \in
\Int(G)\;\; \mbox{such that}\;\;  \sigma'\,=\,\alpha\sigma
\alpha^{-1},
$$

We can define a similar relation $\sim$ in the set $\Aut_2(G)$ of automorphisms of
$G$  of order $2$. 

\begin{remark}
The equivalence relation $\sim$  for elements in $\Aut_2(G)$ should not be
confused with the inner equivalence, meaning the equivalence relation where
two elements are equivalent if they map to the same element in $\Out(G)$. It is easy to
show that if $\theta \,\sim\, \theta'$ then they are inner equivalent.
\end{remark}

Cartan \cite{cartan} shows that there is a bijection  
$$
\mbox{Conj}(G)/\sim \,  \longleftrightarrow\,  \Aut_2(G) /\sim\ .
$$ 
More concretely, one has that given the compact conjugation $\tau$, in each
class $\mbox{Conj}(G)/\sim$ one can find a representative $\sigma$ commuting 
with $\tau$ so that $\theta:=\sigma\tau$ is an element of $\Aut_2(G)$, and
similarly if we start with a class in $\Aut_2(G) /\sim\ $.

%%%%%%%%%%%%%%%%%%%%%%%%%%%%%%%%%%%%%%%%%%%%%%
\subsection{Pseudo-real principal $G$-bundles}
%%%%%%%%%%%%%%%%%%%%%%%%%%%%%%%%%%%%%%%%%%%%%%

We use the notation of Section \ref{realforms-group}.
Let $G$ be a semisimple complex affine algebraic group. Let $\tau\in \Conj(G)$
be a compact conjugation of $G$, and let
$\sigma\,\in\,\mbox{Conj}(G)$ commuting with $\tau$,  and $\theta=\sigma\tau\in
\Aut_2(G)$.

Let $Z^\sigma\,\subset\, Z$ be the fixed point locus in the center $Z\,
\subset G$.  The subgroup of $Z^\sigma$ generated by its elements
of order two will be denoted by $Z^\sigma_2$.

Let $X$ be a compact connected Riemann surface, of genus $g\, \geq\, 2$, equipped with
an anti-holomorphic involution $\alpha:X\lra X$.

\begin{definition}\label{definition-pseudo-real}
Let $E$ be a holomorphic principal $G$-bundle over $X$. Take any
$c\, \in\, Z^\sigma_2$. We say that $E$ is $(\alpha,\sigma,c)$-{\bf pseudo-real} if
$E$ is equipped with
an anti-holomorphic map $\widetilde{\alpha}:E\lra E$  covering $\alpha$ such that 
\begin{itemize}
\item 
$\widetilde{\alpha}(eg)=\tilde{\alpha}(e)\sigma(g),\;\;\mbox{for}\;\; e\in E\;\;
\mbox{ and }\;\; g \in G$.

\item $\widetilde{\alpha}^2(e)\,=\,ec$.
\end{itemize}

If $c=1$, we say that $E$ is $(\alpha,\sigma)$-{\bf real}.
\end{definition}

\begin{remark}\label{order-2}
An alternative definition of pseudo-real bundle allows for $c$ to be any
element of $Z$. However we can modify $\widetilde{\alpha}$ by the action of an
element $a\in Z$ defining a covering map
$\widetilde{\alpha}':=\widetilde{\alpha}.a$. By this, the element $c$ gets
modified by $c'\,=\,a\sigma(a)c$. In particular we can take $a$ lying in
$Z^\sigma$ and the composition is modified by $a^2$. Therefore if $c$ lies in
$(Z^\sigma)^2$, or more generally is of the form $\sigma(a)a$ we can normalize
our pseudo-real structure to a real one.  But since the natural homomorphism
$Z^\sigma_2\,\longrightarrow\, Z^\sigma/(Z^\sigma)^2$ is surjective we can always
assume that $c$ is of order $2$, as we have done in our definition.
\end{remark}

\begin{remark}
Sometimes to emphasize the pseudo-real structure we will write
$(E,\varphi,\widetilde{\alpha})$ for a $G$-Higgs bundle $(E,\varphi)$ 
equipped with a pseudo-real structure $\widetilde{\alpha}$.
\end{remark}
Define the quotient
$$
G_c\, :=\, G/\langle c\rangle \, .
$$
Note that $\langle c\rangle\,=\, {\mathbb Z}/2\mathbb Z$ if $c\,\not=\, 1$.
Since $c$ is fixed by $\sigma$, the involution $\sigma$ induced an
anti-holomorphic involution of $G_c$. This anti-holomorphic involution of $G_c$
will be denoted by $\sigma'$.
Let $(E_G\, , \widetilde{\alpha})$ be a $(\alpha,\sigma,c)$-pseudo-real principal
$G$-bundle on $X$. Define $E_{G_c}\,:=\, E_G/\langle c\rangle$. Note that
$E_{G_c}$ is the principal $G_c$-bundle obtained
by extending the structure group of $E_G$ using the quotient homomorphism
$G\,\longrightarrow\, G_c$. The above self-map $\widetilde{\alpha}$ of $E_G$
descends to a self-map
$$
\widetilde{\alpha}'\, :\, E_{G_c}\, \longrightarrow\, E_{G_c}\, .
$$
Since $\widetilde{\alpha}^2\,=\,c$, we have $\widetilde{\alpha}'
\circ \widetilde{\alpha}'\,=\, \text{Id}_{E_{G_c}}$.
Therefore, $(E_{G_c}\, , \widetilde{\alpha}')$ is a $(\alpha,\sigma')$-real principal
$G_c$-bundle.

The pair $(X\, , \alpha)$ defines a geometrically irreducible smooth projective curve
defined over $\mathbb R$. This curve defined over $\mathbb R$ will be denoted by
$X'$. Assume that $c\,\not=\, 1$. Let $G'$ (respectively, $G'_c$) be the algebraic
group, defined over $\mathbb R$, given by the pair $(G\, ,\sigma)$ (respectively,
$(G_c\, ,\sigma')$). Consider the short exact sequence of sheaves
$$
1\,\longrightarrow\, \langle c\rangle \,=\, {\mathbb Z}/2{\mathbb Z}
\,\longrightarrow\, G' \,\longrightarrow\, G'_c \,\longrightarrow\, 1
$$
on $X'$. Let
$$
H^1_{\rm et}(X',\, G')\,\longrightarrow\, H^1_{\rm et}(X',\, G'_c)
\,\stackrel{\beta}{\longrightarrow}\, H^2_{\rm et}(X',\, {\mathbb Z}/2{\mathbb Z}) 
$$
be the long exact sequence of \'etale cohomologies corresponding to the above
short exact sequence of sheaves on the curve $X'$ defined over $\mathbb R$. As noted
above, a $(\alpha,\sigma,c)$-pseudo-real principal $G$-bundle on $X$ gives
a $(\alpha,\sigma')$-real principal $G_c$-bundle. Note that the isomorphism classes of
principal $G'_c$-bundles on $X'$ are parametrized by the elements of the
cohomology $H^1_{\rm et}(X',\, G'_c)$. Indeed, this follows immediately from the
fact that any principal $G'_c$-bundle on $X'$ can be locally trivialized with
respect to the \'etale topology. Therefore, a $(\alpha,\sigma,c)$-pseudo-real
principal $G$-bundle on $X$ gives an element of $H^1_{\rm et}(X',\, G'_c)$.

We will give a necessary and sufficient condition for a
given $(\alpha,\sigma')$-real principal $G_c$-bundle on $X$ to come from a
$(\alpha,\sigma,c)$-pseudo-real principal $G$-bundle.

Let $(E_{G_c}\, , \widetilde{\alpha}')$ be a $(\alpha,\sigma')$-real principal
$G_c$-bundle on $X$. As explained above, $(E_{G_c}\, , \widetilde{\alpha}')$ is
equivalently a principal $G'_c$-bundle on $X'$. This principal $G'_c$-bundles on $X'$
will be denoted by $F_{G_c}$. Consider the adjoint action of $G$ on itself. Since
$c$ lies in the center of $G$,
this action of $G$ factors through the quotient group $G_c$. Let
$$
E_{G_c}(G)\, :=\, E_{G_c}\times^{G_c} G\, \longrightarrow\, X
$$
be the fiber bundle associated to the principal $G_c$-bundle $E_{G_c}$ for this
action of $G_c$ on $G$. Since the action of $G_c$ on $G$ preserves the group
structure on $G$, each fiber of $E_{G_c}(G)$ is a group isomorphic to $G$.
The action of $G_c$ on $G$ descends to an action of $G_c$ on
the quotient $G/\langle c\rangle\,=\, G_c$, and this descended action
coincides with the adjoint action of $G_c$ on itself. Therefore, the
short exact sequence of groups
$$
1\, \longrightarrow\,{\mathbb Z}/2{\mathbb Z}
\,\longrightarrow\, G \,\longrightarrow\, G_c \,\longrightarrow\, 1
$$
produces a short exact sequence of fiber bundles with group structures
\begin{equation}\label{h1}
1\, \longrightarrow\,X\times ({\mathbb Z}/2{\mathbb Z}) \,\longrightarrow\,
E_{G_c}(G) \,\longrightarrow\, \text{Ad}(E_{G_c}) \,\longrightarrow\, 1 \, ,
\end{equation}
where $\text{Ad}(E_{G_c})\,=\, E_{G_c}\times^{G_c} G_c$ is the adjoint bundle
for $E_{G_c}$.

The involution $\widetilde{\alpha}'$ of $E_{G_c}$ and the involution $\sigma$ of
$G$ together produce an anti-holomorphic involution of $E_{G_c}(G)$ covering
$\alpha$. Similarly, $\widetilde{\alpha}'$ and $\sigma'$
together produce an anti-holomorphic involution of $\text{Ad}(E_{G_c})$ covering
$\alpha$. Therefore, \eqref{h1} produces a short exact sequence
\begin{equation}\label{h11}
1\, \longrightarrow\,X'\times ({\mathbb Z}/2{\mathbb Z}) \,\longrightarrow\,
E_{G_c}(G)' \,\longrightarrow\, \text{Ad}(E_{G_c})' \,\longrightarrow\, 1
\end{equation}
over the curve $X'$ defined over $\mathbb R$. We note that $\text{Ad}(E_{G_c})'$
is the adjoint bundle for the principal $G'_c$-bundle $F_{G_c}$ over $X'$ defined by
the pair $(E_{G_c}\, , \widetilde{\alpha}')$.

The space of all isomorphism classes of principal $G'_c$-bundles on $X'$
are parametrized by $H^1_{\rm et}(X',\, \text{Ad}(E_{G_c})')$. This identification
is constructed as follows. First recall that
$\text{Ad}(E_{G_c})'$ is the adjoint bundle for the principal $G'_c$-bundle
$F_{G_c}$ over $X'$. Given a principal
$G'_c$-bundle on $X'$, by choosing \'etale local isomorphisms of it with
$F_{G_c}$ we get an element of $H^1_{\rm et}(X',\, \text{Ad}(E_{G_c})')$. Conversely,
given a $1$--cocycle on $X'$ with values in $\text{Ad}(E_{G_c})'$, by gluing back,
using the cocycle, the restrictions of $F_{G_c}$ to the open
subsets for the cocycle, we get a principal $G'_c$-bundle on $X'$.
Note that if $F_{G_c}$ is the trivial principal $G'_c$-bundle, then
$H^1_{\rm et}(X',\, \text{Ad}(E_{G_c})')\,=\, H^1_{\rm et}(X',\, G'_c)$.

The set $H^1_{\rm et}(X',\, \text{Ad}(E_{G_c})')$ has a distinguished base point $t_0$.
This point $t_0$ corresponds to the isomorphism class of the
principal $G'_c$-bundle $F_{G_c}$.

Consider the short exact sequence of \'etale cohomologies
\begin{equation}\label{h2}
H^1_{\rm et}(X',\, E_{G_c}(G)')\,\stackrel{\gamma'}{\longrightarrow}\, H^1_{\rm et}(X',
\,  \text{Ad}(E_{G_c})')\,\stackrel{\beta'}{\longrightarrow}\,
H^2_{\rm et}(X',\, {\mathbb Z}/2{\mathbb Z})\,=\, {\mathbb Z}/2{\mathbb Z}
\end{equation}
associated to \eqref{h11}. It can be shown that $(E_{G_c}\, , \widetilde{\alpha}')$
is given by a $(\alpha,\sigma,c)$-pseudo-real principal $G$-bundle if and only if the
base point $t_0\, \in\, H^1_{\rm et}(X', \,  \text{Ad}(E_{G_c})')$ lies in the
image of the map $\gamma'$ in \eqref{h2}. Indeed, if $(E_G\, , \widetilde{\alpha})$ is
a $(\alpha,\sigma,c)$-pseudo-real principal $G$-bundle on $X$ that gives
$(E_{G_c}\, , \widetilde{\alpha}')$, then the adjoint bundle
$\text{Ad}(E_G)$ equipped with the involution constructed using $\widetilde{\alpha}$ and
$\sigma$ produces an element $t'\, \in\, H^1_{\rm et}(X',\, E_{G_c}(G)')$ such that
$\gamma'(t')\,=\, t_0$. Conversely, any $t'\, \in\, H^1_{\rm et}(X',\, E_{G_c}(G)')$
produces a $(\alpha,\sigma,c)$-pseudo-real principal $G$-bundle. If $\gamma'(t')
\,=\, t_0$, then this $(\alpha,\sigma,c)$-pseudo-real principal $G$-bundle gives
the pair $(E_{G_c}\, , \widetilde{\alpha}')$.

Therefore, we have the following.

\begin{proposition}\label{prl}
A $(\alpha,\sigma')$-real principal $G_c$-bundle $(E_{G_c}\, , \widetilde{\alpha}')$ on
$X$ comes from a $(\alpha,\sigma,c)$-pseudo-real principal $G$-bundle if and only
if $\beta'(t_0)\,=\, 0$, where $\beta'$ is the map in \eqref{h2} and
$t_0\, \in\, H^1_{\rm et}(X',\, {\rm Ad}(E_{G_c})')$ is the base point.
\end{proposition}

The following proposition shows the relation between the reality conditions 
defined by conjugations of $G$ that are inner equivalent. One has the 
following.

\begin{proposition}\label{almost-inner}
Let $\sigma,\sigma'\in \Conj(G)$ such that $\sigma'=\Int(g_0)\sigma$ for some
$g_0\in G$, i.e., $\sigma'(g)\,=\,g_0\sigma(g)g_0^{-1}$. Let $E$ be a $G$-bundle
over $X$. Then $E$ is $(\alpha,\sigma,c)$-pseudo-real if and only if it is
$(\alpha,\sigma',c')$-pseudo-real, where $c$ and $c'$ are related by $g_0$ and
$\sigma$. In fact $c'=c$, if $\sigma(g_0)=g_o^{-1}$.
\end{proposition}

\begin{proof}
Let $(E\, ,\widetilde{\alpha})$ be a $(\alpha,\sigma,c)$-pseudo-real principal
$G$-bundle on $X$. Define
$$
\widetilde{\alpha}'\, :\, E\, \longrightarrow\, E\, , ~\ e\, \longmapsto\,
\widetilde{\alpha}(e)g_0^{-1}\, .
$$
Since $\widetilde{\alpha}$ is anti-holomorphic and covers $\alpha$, the map
$\widetilde{\alpha}'$ is also anti-holomorphic and covers $\alpha$. 
For any $e\,\in\, E$ and $g\, \in\, G$, we have
$$
\widetilde{\alpha}'(eg)\,=\, \widetilde{\alpha}(eg)g_0^{-1}\,=
\,\widetilde{\alpha}(e)\sigma(g)g_0^{-1}
\,=\, \widetilde{\alpha}(e)g_0^{-1}g_0\sigma(g)g_0^{-1}\,=\, 
\widetilde{\alpha}'(e)\sigma'(g)\, .
$$
Also,
$$
\widetilde{\alpha}'(\widetilde{\alpha}'(e))\,=\, \widetilde{\alpha}'(\widetilde{\alpha}(e)g_0^{-1})
\,=\, \widetilde{\alpha}(\widetilde{\alpha}(e)g_0^{-1})g_0^{-1}\,=\,
\widetilde{\alpha}(\widetilde{\alpha}(e))\sigma(g_0^{-1})g_0^{-1}
\, =\, ec\sigma(g_0^{-1})g_0^{-1}.
$$
Now, $\sigma'^2=\Id$ implies that  $\sigma(g_0^{-1})g_0^{-1}\in Z$, 
and we can appeal to Remark \ref{order-2} 
to claim that by  modifying
$\widetilde{\alpha}'$ by an element of the center 
 $c\sigma(g_0^{-1})g_0^{-1}$ is replaced by an element $c'\in Z_2^\sigma$,
and hence $E$ has the structure of a $(\alpha,\alpha',c')$-pseudo-real 
principal $G$-bundle on $X$. The last claim in the proposition is clear.
\end{proof}

%Now take any $g_0\, \in\, G$. Let
%$$
%\nu\, :\, G\, \longrightarrow\, G
%$$
%be the anti-holomorphic involution defined by $\nu(g)\, =\, g^{-1}_0\sigma(g)g_0$.
%Note that
%$$
%\nu(z)\, =\, \sigma(z)\, ~\ \forall ~ z\, \in\, Z\, .
%$$
%Therefore, we have $Z^\sigma_2\,=\, Z^\nu_2$.

%\begin{lemma}\label{lef}
%Assume that $\sigma(g_0)\,=\, g^{-1}_0$.
%Take any $c\, \in\, Z^\sigma_2\,=\, Z^\nu_2$. There is a natural identification
%between the $(\alpha,\sigma,c)$-pseudo-real principal $G$-bundles and the
%$(\alpha,\nu,c)$-pseudo-real principal $G$-bundles on $X$.
%\end{lemma}

%\begin{proof}
%Let $(E\, ,\widetilde{\alpha})$ be a $(\alpha,\sigma,c)$-pseudo-real principal $G$-bundle
%on $X$. Define
%$$
%\widetilde{\nu}\, :\, E\, \longrightarrow\, E\, , ~\ e\, \longmapsto\,
%\widetilde{\alpha}(e)g_0\, .
%$$
%Since $\widetilde{\alpha}$ is anti-holomorphic and covers $\alpha$, the map
%$\widetilde{\nu}$ is also anti-holomorphic and covers $\alpha$. For any $e\,\in\, E$
%and $g\, \in\, G$, we have
%$$
%\widetilde{\nu}(eg)\,=\, \widetilde{\alpha}(eg)g_0\,=\,\widetilde{\alpha}(e)\sigma(g)g_0
%\,=\, \widetilde{\alpha}(e)g_0g^{-1}_0\sigma(g)g_0\,=\, \widetilde{\nu}(e)\nu(g)\, .
%$$
%Also,
%$$
%\widetilde{\nu}(\widetilde{\nu}(e))\,=\, \widetilde{\nu}(\widetilde{\alpha}(e)g_0)
%\,=\, \widetilde{\alpha}(\widetilde{\alpha}(e)g_0)g_0\,=\,
%\widetilde{\alpha}(\widetilde{\alpha}(e))\sigma(g_0)g_0\, =\, ec\, .
%$$
%Therefore, $(E\, , \widetilde{\nu})$ is a $(\alpha,\nu,c)$-pseudo-real
%principal $G$-bundle on $X$.
%\end{proof}

%%%%%%%%%%%%%%%%%%%%%%%%%%%%%%%%%%%%%%%%%%%
\subsection{Pseudo-real $G$-Higgs bundles}
%%%%%%%%%%%%%%%%%%%%%%%%%%%%%%%%%%%%%%%%%%%

Let $(E\, , \widetilde{\alpha})$ be a $(\alpha,\sigma,c)$-pseudo-real
principal $G$-bundle on $X$ as defined above. Let
$$
\text{ad}(E)\, :=\, E\times^G {\mathfrak g}\,=:\, E({\mathfrak g})
$$
be the adjoint vector bundle for $E$. The self-map $\widetilde{\alpha}$ of $E$
produces an anti-holomorphic self-map
\begin{equation}\label{involution-ad}
\widetilde{\alpha}_0\, :\, E({\mathfrak g})\,\longrightarrow\, E({\mathfrak g})
\end{equation}
such that $q\circ \widetilde{\alpha}_0\,=\, \alpha\circ q$, where $q$ is the
projection of $E({\mathfrak g})$ to $X$.
Since $c\, \in\, Z$, the adjoint action of $c$ on $\mathfrak g$ is trivial.
This immediately implies that $\widetilde{\alpha}_0$ is an involution. In other words,
$(E({\mathfrak g})\, ,\tilde{\alpha}_0)$ is a real vector bundle
(see \cite{BGH}).

The real structure of the canonical line bundle $K$ of $X$ given by $\alpha$
and the above real structure $\widetilde{\alpha}_0$ of $E({\mathfrak g})$
combine to define a real structure on the vector bundle $E({\mathfrak g})\otimes K$.
For notational convenience, this real structure on $E({\mathfrak g})\otimes K$
will also be denoted by $\widetilde{\alpha}$. So
$$
\widetilde{\alpha}\,:\, E(\lieg)\otimes K\,\lra\, E(\lieg)\otimes K
$$
is an anti-holomorphic involution over $\alpha$.

\begin{definition}
Let  $(E,\varphi)$ be a $G$-Higgs bundle. We say that $(E,\varphi)$ is  
$(\alpha,\sigma,c,+)$-{\bf pseudo-real} 
(respectively, $(\alpha,\sigma,c,-)$-{\bf pseudo-real})
if $E$ is $(\alpha,\sigma,c)$-pseudo-real, and $\varphi$ satisfies 
$$
\widetilde{\alpha}(\varphi)\,=\, \varphi~ \  ~{\rm (respectively},~
\widetilde{\alpha}(\varphi)\,=\, -\varphi{\rm )} \, .
$$
\end{definition}

The concept of $(\alpha,\sigma,c,+)$-{\bf pseudo-real} Higgs bundle was
introduced in  \cite{BGH}, where notions of  (semi)stability and polystability 
for these objects were defined. These notions are identical for
the $(\alpha,\sigma,c,-)$-pseudo-real case.
For the benefit of the reader we recall the basic definitions and 
facts (see \cite{BGH} for details).

Let $\Ad(E):=E\times^G G$ be the group-scheme over $X$ associated to $E$ for
the adjoint action of $G$ on it self. The bundle $\Ad(E)$ is equipped with an
anti-holomorphic involution 
\begin{equation}\label{involution-Ad}
\widetilde{\alpha}: \Ad(E)\,\longrightarrow\, \Ad(E)
\end{equation}
 (abusing
notation again) covering $\alpha$. Note that $\widetilde{\alpha}^2\,=\,
\Id_{\Ad(E)}$ since the adjoint action of $Z^\sigma$ on $G$ is trivial.

A {\bf  parabolic subgroup scheme} of $\Ad(E)$ is a Zariski
closed analytically locally trivial subgroup scheme 
$\underline{P}\subset \Ad(E)$ such that $\Ad(E)/\underline{P}$ is compact. 
For such a parabolic
subgroup scheme $\underline{P}$ let $\underline{\liep}\subset \ad(E)$ be the
corresponding bundle of  Lie algebras. 

A  $(\alpha,\sigma,c,\pm)$-{\bf pseudo-real}  $G$-Higgs bundle 
$(E,\varphi,\widetilde{\alpha})$ is {\bf semistable} 
(respectively {\bf stable})
if for every proper parabolic subgroup scheme $\underline{P}\subset \Ad(E)$ 
such that
$\widetilde{\alpha}(\underline{P})\subset \underline{P}$, where 
$\widetilde{\alpha}$ is given by
(\ref{involution-Ad}), and $\varphi\in H^0(X,\underline{\liep}\otimes K)$,
$$
\deg(\underline{\liep})\leq 0\;\;\; (\mbox{respectively},\,\,\, 
\deg(\underline{\liep}) < 0,
$$
where $\underline{\liep}$ is the vector bundle associated to $\underline{P}$ defined above.
  
One has the following (see \cite{BGH}).

\begin{proposition}\label{stability-versus-real-stability}
Let  $(E,\varphi,\widetilde{\alpha})$ 
be a  $(\alpha,\sigma,c,\pm)$-{\bf pseudo-real} $G$-Higgs bundle.

(1) If $(E,\varphi)$ is semistable (respectively stable), in the sense of
Section \ref{moduli}, then $(E,\varphi,\widetilde{\alpha})$ is 
semistable (respectively stable).

(2) If $(E,\varphi,\widetilde{\alpha})$ is  semistable then $(E,\varphi)$ 
is semistable.

(3) If $(E,\varphi,\widetilde{\alpha})$ is  stable then $(E,\varphi)$ 
is polystable (in the sense given in Section \ref{moduli}.
\end{proposition}

To define polystability for a pseudo-real $G$-Higgs bundle let 
$\underline{\liep}\subset \ad(E)$ be a parabolic subalgebra bundle such that
$\widetilde{\alpha_0}(\underline{\liep})=\underline{\liep}$, 
where $\widetilde{\alpha_0}$ is the 
involution defined in (\ref{involution-ad}).

Let $R_u(\underline{\liep})\subset \underline{\liep}$ be the holomorphic subbundle over $X$ whose
fiber over a point $x\in X$ is the nilpotent radical of the parabolic 
subalgebra $\liep_x$. Therefore, the quotient 
$\underline{\liep}/R_u(\underline{\liep})$ is a 
bundle of reductive Lie algebras. Note that 
$\widetilde{\alpha_0}(R_u(\underline{\liep}))=R_u(\underline{\liep})$. 
A {\bf Levi subalgebra bundle} of
$\underline{\liep}$ is a holomorphic subbundle
$$
\ell({\underline{\liep}})\, \subset\, \underline{\liep}
$$
such that for each $x\, \in\, X$, the fiber
$\ell(\underline{\liep})_x$ is a Lie subalgebra of ${\underline{\liep}}_x$, and the
composition
$$
\ell(\underline{\liep})\, \hookrightarrow\, \underline{\liep}\, \longrightarrow\,
\underline{\liep}/R_u(\underline{\liep})
$$
is an isomorphism, where $\underline{\liep}\, \longrightarrow\,
\underline{\liep}/R_u(\underline{\liep})$ is the quotient map.

A semistable $(\alpha,\sigma,c,\pm)$-pseudo-real $G$-Higgs bundle   $(E,\varphi,\widetilde{\alpha})$ 
is {\bf polystable} if either is stable, or there is a proper parabolic subalgebra bundle 
$\underline{\liep}\,\subsetneq\,
{\ad}(E)$, and a Levi subalgebra bundle $\ell(\underline{\liep})\, \subset\,
\underline{\liep}$, such that $\widetilde{\alpha_0}(\underline{\liep})\, =\, 
\underline{\liep}$,
$\widetilde{\alpha_0}(\ell(\underline{\liep}))\, =\, \ell(\underline{\liep})$, 
$\varphi\in H^0(X,\ell(\underline{\liep})\otimes K)$, and for every parabolic
subalgebra bundle $\underline{\lieq}\subset \ell(\underline{\liep})$ with
$\widetilde{\alpha_0}(\underline{\lieq})\, =\, \underline{\lieq}$ we have 
$$
\deg(\underline{\lieq}) < 0.
$$ 

We have the following (see \cite{BGH}).

\begin{proposition} \label{real-polystability-versus-polystability}
A $(\alpha,\sigma,c,\pm)$-pseudo-real $G$-Higgs bundle   $(E,\varphi,\widetilde{\alpha})$ 
is  polystable if and only if $(E,\varphi)$ is polystable.
\end{proposition}

We can thus define the moduli space $\cM(G,\alpha,\sigma,c,\pm)$
of isomorphism classes of polystable $(\alpha,\sigma,c,\pm)$-pseudo-real 
$G$-Higgs bundles, and,  as a consequence of
Proposition \ref{real-polystability-versus-polystability},
define  maps 
\begin{equation}\label{fo1}
\cM(G,\alpha,\sigma,c,\pm)\,\lra\, \cM(G)
\end{equation}
that forget the pseudo-real structure.

%%%%%%%%%%%%%%%%%%%%%%%%%%%%%%%%%%%%%%
\section{Involutions of moduli spaces}
%%%%%%%%%%%%%%%%%%%%%%%%%%%%%%%%%%%%%%

%%%%%%%%%%%%%%%%%%%%%%%%%%%%%%%%%%%%%%%%%%%%%%%%%%%%%%%%%%%%%%%%%%%%%%%%%%%%%%
\subsection{Involutions of Higgs bundle moduli spaces}\label{section-involutions}
%%%%%%%%%%%%%%%%%%%%%%%%%%%%%%%%%%%%%%%%%%%%%%%%%%%%%%%%%%%%%%%%%%%%%%%%%%%%%%

As before, let $\alpha\,:\,X\,\lra\, X$ and  $\sigma\,:\,G\,\lra\, G$ be anti-holomorphic
involutions. For a holomorphic principal $G$-bundle $E$ on $X$, let $\sigma(E)$ be the
$C^\infty$ principal $G$-bundle on $X$ obtained by extending the structure group of $E$
using the homomorphism $\sigma$. So the total space of $\sigma(E)$ is identified with
that of $E$, but the action of $g\, \in\, G$ on $e\, \in\, E$ coincides with the
action of $\sigma(g)$ on $e\, \in\, \sigma(E)$. Consequently, the pullback
$\alpha^*\sigma(E)$ has a holomorphic structure given by the holomorphic structure of
$E$. Let
$$
\widetilde{\sigma}\, :\, E({\mathfrak g})\, \longrightarrow\,E({\mathfrak g})
$$
be the conjugate linear isomorphism that sends the equivalence class of
any $(e\, , v)\, \in\, E\times {\mathfrak g}$ to the equivalence class of $(e\, ,
d\sigma (v))$, where $d\sigma$ is the automorphism of $\mathfrak g$ corresponding to
$\sigma$. Let $\varphi$ be a Higgs field on $E$. Let $\sigma(\varphi)$ be the
$C^\infty$ section of $E({\mathfrak g})\otimes K$ defined by $\widetilde{\sigma}$
and the $C^\infty$ isomorphism $K\, \longrightarrow\, \overline{K}$ defined by
$df\, \longmapsto\, d\overline{f}$, where $f$ is any locally defined holomorphic function
on $X$.

We have involutions
\begin{equation}\label{involution-ab}
   \begin{aligned}
\iota(\alpha,\sigma)^{\pm}: \cM(G) & \,\longrightarrow\, \cM(G) \\
(E,\varphi) & \longmapsto\, (\alpha^*\sigma(E)\, ,\pm \alpha^\ast\sigma(\varphi)).
  \end{aligned}
\end{equation}

\begin{proposition}\label{p1}
The image of the map $$\cM(G,\alpha,\sigma,c,+)\,\lra\, \cM(G)$$
in \eqref{fo1} is contained in the
fixed point  locus of the involution $\iota(\alpha,\sigma)^{+}$. Moreover, the
fixed point locus of $\iota(\alpha,\sigma)^{+}$ in the smooth locus $\cM(G)^{\rm sm}
\,\subset\, \cM(G)$ is the intersection of $\cM(G)^{\rm sm}$ with the 
union of the images of
$\cM(G,\alpha,\sigma,c,+)$ as $c$ runs over $Z^\sigma_2$, where $Z^\sigma_2$ as before
is the subgroup of $Z^\sigma$ generated by the order two points.

Similarly, the fixed point locus of $\iota(\alpha,\sigma)^{-}$ in $\cM(G)^{\rm sm}$
is the intersection of $\cM(G)^{\rm sm}$ with the  union of the images of
$\cM(G,\alpha,\sigma,c,-)$ as $c$ runs over $Z^\sigma_2$.
\end{proposition}

\begin{proof}
{}From the definition of $\iota(\alpha,\sigma)^{+}$ (respectively,
$\iota(\alpha,\sigma)^{-}$) it follows immediately that $\cM(G,\alpha,\sigma,c,+)$
(respectively, $\cM(G,\alpha,\sigma,c,-)$) is contained in the fixed point locus
of $\iota(\alpha,\sigma)^{+}$ (respectively, $\iota(\alpha,\sigma)^{-}$).

A $G$-Higgs bundle $(E\, ,\varphi)$ lies in $\cM(G)^{\rm sm}$ 
if $(E\, ,\varphi)$ is
stable and the automorphism group of $(E\, ,\varphi)$ coincides with the
center $Z$
of $G$, i.e., if the Higgs bundle is simple as defined in
\cite{garcia-prada-gothen-mundet,garcia-prada,BH}
(we recall that such bundles are called regularly stable). 
Suppose that $(E\, ,\varphi)\in \cM(G)^{\rm sm}$ is fixed under the involution 
$\iota(\alpha,\sigma)^{+}$ (respectively $\iota(\alpha,\sigma)^{-}$). This
means that there exists an isomorphism
$$
f\, :\, E\,\longrightarrow\, \alpha^*\sigma(E)
$$
such that $\alpha^*\sigma(f)\circ f\in \Aut(E,\varphi)$, 
but since $\Aut(E,\varphi)= Z$, we have that 
$\alpha^*\sigma(f)\circ f=c\in Z$.
We can interpret $f$ as a map $f'\,:\, E\, \too\, \sigma(E)$ such that
$\sigma(f')\circ f'=c\in Z$.
Identifying $\sigma(E)$ with $E$ with multiplication on 
the right defined by $e\cdot g= e\sigma(g)$, where $g\in G$ and $e\in E$,
we are indeed defining  a $(\alpha,\sigma,c,+)$ (respectively $(\alpha,\sigma,c,-)$)
pseudo-real structure on $(E,\varphi)$, since we can always assume that $c\in Z_2^\sigma$,
as explained in Remark \ref{order-2}. In other words, $(E\, ,\varphi)$ lies in the
image of $\cM(G,\alpha,\sigma,c,+)$ (respectively $\cM(G,\alpha,\sigma,c,-)$).
\end{proof}

\begin{remark}
In Definition \ref{definition-pseudo-real} we could have  defined a pseudo-real structure
replacing $\widetilde{\alpha}$ by  an anti-holomorphic map $\widetilde{\alpha}'\,:
\, E\,\longrightarrow\, \sigma(E)$ of $G$--bundles
covering $\alpha$. Although $\sigma(E)$ is no longer a holomorphic
bundle, its total space is a complex manifold because it is identified with the
total space of $E$, and hence the
anti-holomorphicity condition makes sense. The condition 
$\widetilde{\alpha}(eg)\,=\,\widetilde{\alpha}(e)\sigma(g)$ 
in Definition \ref{definition-pseudo-real} 
is now automatic since  $\widetilde{\alpha}'$ is a $G$--bundle map.
\end{remark}

\begin{proposition}\label{proposo1}
Let $\sigma$ and $\sigma'$ be inner equivalent elements in $\Conj(G)$,
i.e., they define the same element in $\Out_2(G)$. Then
$$
\iota(\alpha,\sigma)^{+}\,=\, \iota(\alpha,\sigma')^{+}\, \
(respectively,\,~  \iota(\alpha,\sigma)^{-}\,=\, \iota(\alpha,\sigma')^{-}\,).
$$
\end{proposition}

\begin{proof}
If we replace $\sigma$ by $\sigma'\, :=\, g_0
\sigma g_0^{-1}$, where $g_0\, \in\, G$, then the corresponding anti-holomorphic involution of the
moduli space is replaced by its composition with the holomorphic automorphism 
of the moduli space corresponding to the automorphism of $G$ defined by
$g\,\longmapsto\, g_0g g_0^{-1}$. But this automorphism of $G$ produces the identity map
of the moduli space. Therefore, the anti-holomorphic involution of the
moduli space is unchanged if $\sigma$ is replaced by $\sigma'$.
\end{proof}

\begin{remark}
Consider the identification
between the $(\alpha,\sigma,c)$-pseudo-real principal $G$-bundles and the
$(\alpha,\sigma',c')$-pseudo-real principal $G$-bundles on $X$ given by  Proposition
\ref{almost-inner} when $\sigma$ and $\sigma'$ are inner equivalent. Note that
a Higgs  field on a $(\alpha,\sigma',c')$-pseudo-real principal
$G$-bundle produces a Higgs field on the corresponding $(\alpha,\sigma,c)$-pseudo-real
principal $G$-bundle, and vice versa. We thus have that by Proposition \ref{almost-inner}
$\cM(G,\alpha,\sigma,c,+)$ is isomorphic to $\cM(G,\alpha,\sigma',c',+)$
(respectively  $\cM(G,\alpha,\sigma,c,-)$ is isomorphic to $\cM(G,\alpha,\sigma',c',-)$)
giving the same image under the corresponding maps to $\cM(G)$.
\end{remark}

%%%%%%%%%%%%%%%%%%%%%%%%%%%%%%%%%%%%%%%%%%%%%%%%%%%%%%%%%%%%%%%%%%%%%%%%%%%%%%
\subsection{Correspondence with representations for $\iota(\alpha,\sigma)^+$}
%%%%%%%%%%%%%%%%%%%%%%%%%%%%%%%%%%%%%%%%%%%%%%%%%%%%%%%%%%%%%%%%%%%%%%%%%%%%%%

We have the orbifold fundamental group of $(X,\alpha)$ that we will
denote $\Gamma(X,\alpha)$ (see \cite{BGH} for example). This fits into an exact sequence
\begin{equation}\label{orbifold-pi1}
1\,\lra\, \pi_1(X, x_0) \,\lra\, \Gamma(X,x_0) \,\lra\, \Z/2\Z \lra 1\, .
\end{equation}

Let $\text{Map}'(\Gamma(X,x_0)\, , G \times ({\mathbb Z}/2{\mathbb Z}))$ be the
space of all maps
$$
\delta\, :\, \Gamma(X,x_0)\,\longrightarrow\,G \times ({\mathbb Z}/2{\mathbb Z})
$$
such that the following diagram is commutative:
\begin{equation}\label{dil}
\begin{matrix}
1 & \longrightarrow & \pi_1(X,x_0)&\longrightarrow &\Gamma(X,x_0) &
\stackrel{\eta}{\longrightarrow} &
{\mathbb Z}/2{\mathbb Z} &\longrightarrow & 1\\
&& \Big\downarrow && ~\Big\downarrow\delta && \Vert\\
1 & \longrightarrow & G &\longrightarrow & G \times ({\mathbb Z}/2{\mathbb Z}) &
\longrightarrow &
{\mathbb Z}/2{\mathbb Z} &\longrightarrow & 1.
\end{matrix}
\end{equation}

Take an element $c\, \in\, Z^\sigma_2$ in the subgroup generated by the elements of
$Z^\sigma$ order two. Using it, we will define another group structure
on $G \times ({\mathbb Z}/2{\mathbb Z})$. The group operation is given by
$$
(g_1\, ,e_1)\cdot (g_2\, ,e_2)\,=\, (g_1(\sigma)^{e_1}(g_2)c^{e_1e_2}\, ,
e_1+e_2)\, .
$$
Note that when $c=1$ we obtain a semidirect product.

Let
$\Hom_c(\Gamma(X,x_0)\, , G \times ({\mathbb Z}/2{\mathbb Z})$ be the
space of all maps
$$
\delta\, \in\, \text{Map}'(\Gamma(X,x_0)\, , G \times ({\mathbb Z}/2{\mathbb Z}))
$$
such that $\delta$ is a homomorphism with respect to this group structure.

Two elements $\delta\, ,\delta'\, \in\, \Hom_c(\Gamma(X,x_0)\, ,
G \times ({\mathbb Z}/2{\mathbb Z}))$ are called {\bf equivalent} if there is an
element $g\, \in\, G$ such
that $\delta'(z)\,=\, g^{-1}\delta(z)g$ for all $z\, \in\, \pi_1(X,\alpha)$.

\begin{theorem}\label{thm1}
The moduli space $\cM(G,\alpha,\sigma,c,+)$ is identified with the space of equivalence
classes of reductive elements of $\Hom_c(\Gamma(X,x_0)\, ,
G \times ({\mathbb Z}/2{\mathbb Z}))$.
\end{theorem}

\begin{proof}
This follows from Proposition 5.6 of \cite{BGH}.
\end{proof}

%Since the fixed point locus of the involution $\iota(\alpha,\sigma)^+$ is given by the
%union of all $\cM(G,\alpha,\sigma,c,+)$, from Theorem \ref{thm1} we get the
%following:

%\begin{corollary}\label{corvn1}
%Consider the involution $\iota(\alpha,\sigma)^+$ of the smooth locus
%of the moduli space $\cM(G)$. Under the correspondence given in Theorem
%\ref{thm1}, the fixed point locus of it is contained in the
%union of  all $\Hom_c(\pi_1(X,\alpha)\, ,
%G \times ({\mathbb Z}/2{\mathbb Z}))$ (modulo the equivalence relation defined
%above) as $c$ runs over $Z^\sigma_2$.
%\end{corollary}

\begin{theorem}\label{thm0a}
Consider the involution $\iota(\alpha,\sigma)^+$ of $\cM(G)$. It is anti-holomorphic with
respect to the almost complex structures $J_1$ and $J_2$,  and it is holomorphic with
respect to $J_3$.
\end{theorem}

\begin{proof}
The almost complex structure $J_1$ is the almost complex structure of the Dolbeault
moduli space (the moduli space of Higgs bundles). Therefore, $\iota(\alpha,\sigma)^+$
is anti-holomorphic with respect to $J_1$.

The almost complex structure $J_2$ is the almost complex structure of the Betti
moduli space (the representation space $\calR(G)$). Note  that
the almost complex structure of the Betti    
moduli space coincides with that of the de Rham moduli space.

As before, fix a base point $x_0\, \in\, X$. The involution $\alpha$ of $X$ produces an
isomorphism
$$
\alpha'\, :\, \pi_1(X, x_0)\, \longrightarrow\, \pi_1(X, \alpha(x_0))\, .
$$
This in turn gives a biholomorphism
$$
\alpha''\, :\,  \Hom^{+}(\pi_1(X,x_0),G) / G \, \longrightarrow\, \Hom^{+}(\pi_1(X,
\alpha(x_0)),G) / G\, .
$$
As noted before, $\calR(G) \,=\, \Hom^{+}(\pi_1(X,x_0),G) / G$ is independent
of the choice of the base point. So $\alpha''$ is a biholomorphism
\begin{equation}\label{bhm}
\alpha''\, :\, \calR(G)\,\longrightarrow\, \calR(G)\, .
\end{equation}
Since $\alpha$ is an involution, it follows that $\alpha''$ is also an involution.

Let
$$
b\, :\, \calR(G) \,=\, \Hom^{+}(\pi_1(X,x_0),G) / G\,\longrightarrow\,
\calR(G)
$$
be the anti-holomorphic involution defined by $\rho\, \longmapsto\, \sigma\circ\rho$.
In other words, $b$ sends a homomorphism $\rho\, :\, \pi_1(X)\,\longrightarrow\, G$
to the composition
$$
\pi_1(X,x_0)\,\stackrel{\rho}{\longrightarrow}\, G\,\stackrel{\sigma}{\longrightarrow}
\, G\, .
$$
Clearly $b$ commutes with the above involution $\alpha''$ in \eqref{bhm}. Therefore,
$b\circ\alpha''$ is also an involution. Note that $b\circ\alpha''$ is anti-holomorphic
because $\alpha''$ is holomorphic and $b$ is anti-holomorphic.

The above involution $b\circ\alpha''$ of $\calR(G)$ coincides with the
involution $\iota(\alpha,\sigma)^+$ of $\cM(G)$ under the correspondence
$\cM(G)\cong \calR(G)$.
Therefore, $\iota(\alpha,\sigma)^+$ is anti-holomorphism with respect to $J_2$.

We recall that $J_3\,=\,J_1J_2$. Since $\iota(\alpha,\sigma)^+$
is anti-holomorphic with respect to both $J_1$ and $J_2$, from the
above identity it follows immediately that $\iota(\alpha,\sigma)^+$ is
holomorphic with respect to $J_3$.
\end{proof}

Since $\calR(G)$ is hyper-K\"ahler, the holomorphic symplectic form $\Omega_2$ on it 
is flat with respect to the K\"ahler structure $\omega_2$ corresponding to $J_2$.
Similarly, the holomorphic symplectic form $\Omega_1$ with respect to $J_1$
is flat with respect to the K\"ahler structure $\omega_1$ corresponding to $J_1$.
In particular, $\calR(G)$ and $(\cM(G)\, ,J_1\, , \omega_1\, ,\Omega_1)$ are Calabi-Yau.

\begin{theorem}\label{thm1a}
The moduli space $\cM(G,\alpha,\sigma,c,+)$ is a special Lagrangian subspace
of $\calR(G)$. Similarly, it is special Lagrangian with respect to
$(\cM(G)\, ,J_1\, , \omega_1\, , \Omega_1)$. Also, it is complex Lagrangian
with respect to $(J_3\, , \Omega_3)$.
\end{theorem}

\begin{proof}
Since the involution $\iota(\alpha,\sigma)^+$ is holomorphic with
respect to $J_3$, it follows that $\cM(G,\alpha,\sigma,c,+)$ is a
holomorphic subspace with respect to $J_3$. Recall that $\Omega_3\,=\,
\omega_1+\sqrt{-1}\omega_2$. The involution $\iota(\alpha,\sigma)^+$
is anti-holomorphic with respect to $J_1$ and $J_2$ and it is an isometry.
Hence $\iota(\alpha,\sigma)^+$ takes $\omega_1$ and $\omega_2$ to
$-\omega_1$ and $-\omega_2$ respectively. Hence $\iota(\alpha,\sigma)^+$ takes
$\Omega_3$ to $-\Omega_3$. This immediately implies that
$\cM(G,\alpha,\sigma,c,+)$ is Lagrangian with respect to $\Omega_3$.

Since $\cM(G,\alpha,\sigma,c,+)$ is the fixed
point locus of an isometric anti-holomorphic involution of the
Calabi-Yau space $\calR(G)$, it follows
that $\cM(G,\alpha,\sigma,c,+)$ is a special Lagrangian subspace
of $\calR(G)$. For a similar reason, $\cM(G,\alpha,\sigma,c,+)$ is a special
Lagrangian subspace of $(\cM(G)\, ,J_1\, , \omega_1)$.
\end{proof}

%%%%%%%%%%%%%%%%%%%%%%%%%%%%%%%%%%%%%%%%%%%%%%%%%%%%%%%%%%%%%%%%%%%%%%%%%%%%%%
\subsection{Correspondence with representations for $\iota(\alpha,\sigma)^-$}
%%%%%%%%%%%%%%%%%%%%%%%%%%%%%%%%%%%%%%%%%%%%%%%%%%%%%%%%%%%%%%%%%%%%%%%%%%%%%%
Next we consider the involution $\iota(\alpha,\sigma)^-$.

Consider the holomorphic involution  $\theta=\sigma\tau$ of $G$ as 
defined earlier in Section \ref{realforms-group}.
Using $c\,\in\, Z^\sigma_2$, we will define yet another group structure
on $G \times ({\mathbb Z}/2{\mathbb Z})$. The group operation is given by
$$
(g_1\, ,e_1)\cdot (g_2\, ,e_2)\,=\, (g_1(\theta)^{e_1}(g_2)c^{e_1e_2}\, ,
e_1+e_2)\, .
$$
Let
$\Hom^{-}_c(\pi_1(X,\alpha)\, , G \times ({\mathbb Z}/2{\mathbb Z}))$ be the
space of all maps
$$
\delta\, \in\, \text{Map}'(\pi_1(X,\alpha)\, , G \times ({\mathbb Z}/2{\mathbb Z}))
$$
such that $\delta$ is a homomorphism with respect to this new group structure.

Two elements $\delta'\, ,\delta'\, \in\, \Hom^{-}_c(\pi_1(X,\alpha)\, ,
G \times ({\mathbb Z}/2{\mathbb Z}))$ are called \textit{equivalent} if there is an
element $g\, \in\, G$ such
that $\delta'(z)\,=\, g^{-1}\delta(z)g$ for all $z\, \in\, \pi_1(X,\alpha)$.

\begin{theorem}\label{thm2}
The moduli space $\cM(G,\alpha,\sigma,c,-)$ is identified with the space of equivalence
classes of reductive elements of $\Hom^{-}_c(\pi_1(X,\alpha)\, ,
G \times ({\mathbb Z}/2{\mathbb Z}))$.
\end{theorem}

\begin{proof}
This follows from Proposition 5.6 of \cite{BGH}.
\end{proof}

%We have the following analog of Corollary \ref{corvn1}.

%\begin{corollary}\label{corvn2}
%Consider the involution $\iota(\alpha,\sigma)^-$ of the smooth locus
%of the moduli space $\cM(G)$. The fixed point locus of it is contained in the
%union of all $\Hom^{-}_c(\pi_1(X,\alpha)\, , G \times ({\mathbb Z}/2{\mathbb Z}))$
%as $c$ runs over $Z^\sigma_2$.
%\end{corollary}

\begin{theorem}\label{thm2a}
Consider the involution $\iota(\alpha,\sigma)^{-}$ of $\cM(G)$. It is anti-holomorphic with
respect to the almost complex structures $J_1$ and $J_3$, and it is holomorphic with
respect to $J_2$.
\end{theorem}

\begin{proof}
The involution $\iota(\alpha,\sigma)^-$ is clearly anti-holomorphic with respect to $J_1$
because $J_1$ coincides with the complex structure of the Dolbeault moduli space.

Let
$$
\widetilde{b}\, :\, \calR(G) \,=\, \Hom^{+}(\pi_1(X,x_0),G) / G\,\longrightarrow\,
\calR(G)
$$
be the holomorphic involution defined by $\rho\, \longmapsto\, \theta\circ\rho$.
In other words, $\widetilde{b}$ sends a homomorphism $\rho\, :\, \pi_1(X)\,\longrightarrow
\, G$ to the composition
$$
\pi_1(X)\,\stackrel{\rho}{\longrightarrow}\, G\,\stackrel{\theta}{\longrightarrow}
\, G\, .
$$
Clearly $\widetilde{b}$ commutes with the above involution $\alpha''$ in \eqref{bhm}.
Therefore, $\widetilde{b}\circ\alpha''$ is also an involution. The
composition $\widetilde{b}\circ\alpha''$ is holomorphic
because both $\alpha''$ and $\widetilde{b}$ are holomorphic.

The above involution $\widetilde{b}\circ\alpha''$ of $\calR(G)$ coincides with
$\iota(\alpha,\sigma)^-$, and the complex structure of the Betti moduli space
$\calR(G)$ is given by $J_2$.
Therefore, $\iota(\alpha,\sigma)^-$ is holomorphic with respect to $J_2$.

Since $J_3\,=\,J_1J_2$, and $\iota(\alpha,\sigma)^-$
is anti-holomorphic with respect to $J_1$ and holomorphic
with respect to $J_2$, we conclude that $\iota(\alpha,\sigma)^-$ is
anti-holomorphic with respect to $J_3$.
\end{proof}

Consider the complex structure $J_3$ and the corresponding holomorphic symplectic
form $\Omega_3$. Since $\calR(G)$ is hyper-K\"ahler, $\Omega_3$ is flat with respect
to the K\"ahler structure for $J_3$. Now we have following analog of Theorem
\ref{thm1a}.

\begin{theorem}\label{thm1b}
The moduli space $\cM(G,\alpha,\sigma,c,-)$ is a special Lagrangian subspace
of $(\cM(G)\, ,J_1\, , \omega_1\, ,\Omega_1)$. Similarly, it is special Lagrangian with
respect to $(\cM(G)\, ,J_3\, , \omega_3\, ,\Omega_3)$. Also, it is a complex
Lagrangian subspace with respect to $(\calR(G)\, ,J_2\, ,\Omega_2)$.
\end{theorem}

\begin{corollary}\label{cor1}
The fixed point locus of the involution $\iota(\alpha,\sigma)^-$ is a complex
subspace of $\cM(G)$ with the complex structure induced by $J_2$, i.e., the
natural complex structure of the moduli space of representations $\calR(G)$.
\end{corollary}

\begin{remark}
Corollary \ref{cor1} is obtained by Baraglia--Schaposnik
\cite{baraglia-schaposnik} in the case when
$\sigma$ is the compact conjugation $\tau$.
\end{remark}

\section*{Acknowledgements}

The main results of this paper were presented  at the workshop
on Higgs bundles and pressure metrics held in Aarhus in August 2013. We wish
to thank Joergen Andersen and the Centre for Quantum Geometry of Moduli Spaces 
for the  invitation and hospitality.  We want to thank Steve Bradlow and Laura
Schaposnik for useful discussions.

\end{document}